\documentclass[12pt]{amsproc}

\usepackage[T1]{fontenc}
\usepackage[utf8]{inputenc}
\usepackage{lmodern}
\usepackage{microtype}
\usepackage{amsmath,amsthm,amssymb,mathtools, todonotes}
\usepackage{enumitem}
\usepackage{hyperref}
\usepackage[
  a4paper,
  left=28mm,
  right=28mm,
  top=28mm,
  bottom=30mm
]{geometry}

\title[$L^2$-cohomology and deformations]{$L^2$-cohomology and deformations of \\ the left regular representation}
\author{Christoph Gamm}
\author{Andreas Thom}
\address{TU Dresden, Fakult\"at Mathematik, 01062 Dresden, Germany}
\email{andreas.thom@tu-dresden.de}

\date{April 13, 2026}
\subjclass[2020]{20J06, 22D10, 47B10, 58D19}
\keywords{unitary representation, deformation theory, group cohomology, Hilbert-Schmidt operators, $\ell^2$-cohomology}

\theoremstyle{plain}
\newtheorem{theorem}{Theorem}[section]
\newtheorem{proposition}[theorem]{Proposition}
\newtheorem{corollary}[theorem]{Corollary}
\newtheorem{lemma}[theorem]{Lemma}
\newtheorem*{theoremA}{Theorem A}
\newtheorem*{theoremB}{Theorem B}
\theoremstyle{definition}
\newtheorem{definition}[theorem]{Definition}
\newtheorem{remark}[theorem]{Remark}

\numberwithin{equation}{section}
\newcommand{\HS}{{\rm HS}}
\newcommand{\BH}{B(\mathcal H)}

\newcommand{\N}{\mathbb{N}}
\newcommand{\La}{\lambda}
\newcommand{\ol}[1]{\overline{#1}}

\DeclareMathOperator{\Ad}{Ad}
\DeclareMathOperator{\Aut}{Aut}
\DeclareMathOperator{\Confdim}{Confdim}
\DeclareMathOperator{\arccosh}{arccosh}

\begin{document}

\begin{abstract}
We study deformations of unitary representations $\pi\colon G\to U(\mathcal H)$ whose coefficients lie in the Hilbert-Schmidt ideal $\HS(\mathcal H)\subset \BH$. Interesting applications arise for the left-regular representation of surface groups and, more generally, cocompact lattices in the automorphism group of a Fuchsian building of conformal dimension $<2$.
\end{abstract}

\maketitle

\tableofcontents

\section{Introduction}
Deformation theory for representations is governed by group cohomology: first-order deformations are described by $1$-cocycles, and higher-order obstructions live in degree two. In this paper we carry out that program for unitary representations of discrete groups when the deformation coefficients are required to lie in the Hilbert-Schmidt ideal.

For finite-dimensional representation varieties this perspective goes back at least to Weil's cohomological description of infinitesimal deformations and local rigidity~\cite{Weil1964}. More broadly, deformation-theoretic ideas in algebra were developed by Gerstenhaber~\cite{Gerstenhaber1964}, and deformation theory for modules was studied by Yau~\cite{Yau2005}. Here we isolate an analogous analytic mechanism for unitary representations on Hilbert space, with Hilbert-Schmidt coefficients and explicit convergence estimates. This research is part of the forthcoming PhD thesis by the first author.

\medskip

Let $\pi\colon G\to U(\mathcal H)$ be a unitary representation. The conjugation action
$
g\cdot T=\pi(g)T\pi(g)^{-1}
$
makes the Hilbert-Schmidt ideal $\HS(\mathcal H)$ into a unitary $G$-module. In the regular case $\pi=\La$, the Fell absorption principle identifies $\HS(\ell^2 G)_{\Ad\La}$ with $\ell^2(G)\otimes \mathcal K$ for a separable multiplicity space $\mathcal K$ carrying the trivial action. This connects Hilbert-Schmidt deformations of the left regular representation directly with $\ell^2$-cohomology. It is this connection that we want to exploit to produce examples of nontrivial deformations. The main result of this paper is the construction of an $n$-parameter analytic unitary deformation of the left regular representation that takes as input an $n$-tuple of tangential directions $c_1,\dots,c_n \in Z^1(G,\HS(\ell^2 G))$ and the vanishing of the \emph{unreduced} second cohomology group with $\ell^2$-coefficients in order to deal with the arising obstruction classes, see Theorem \ref{thm:unitaryanalytic} and Corollary \ref{cor:regular}. The construction of analytic deformations proceeds in two steps: first, we construct a formal power series solution to the deformation problem, and then we show that the formal solution converges in the Hilbert-Schmidt norm. The convergence argument is based on uniform estimates for primitives of cocycles, which are obtained from the open mapping theorem for Fr\'echet spaces in presence of vanishing of unreduced cohomology.

Our first main results can be summarized as follows:

\begin{theoremA}[see Theorem \ref{thm:unitaryanalytic}] Let $G$ be a finitely generated group and $\pi \colon G \to U(\mathcal H)$ be a unitary representation,
assume that $c_{1},\dots,c_{n}\in Z^1(G,\HS(\mathcal H)_{\Ad\pi})$ are skew-adjoint and that $H^2(G,\HS(\mathcal H)_{\Ad\pi})=0$.
Then there exists an $n$-parameter analytic unitary deformation of $\pi$ whose first derivative at $0$ is given by $c_1,\dots,c_n$.
\end{theoremA}

This theorem can be proven without any obstruction theory when all cocycles take values in the abelian subalgebra of diagonal operators, see the end of Section \ref{sec:motivation}; however, in the general non-abelian setting, the obstruction classes can be nontrivial and the vanishing of $H^2(G,\ell^2(G))$ is, at least a priori, necessary to obtain the desired deformations.

Interesting applications arise for example for surface groups or, more generally, cocompact lattices in the automorphism group of a Fuchsian building of conformal dimension less than $2$. The second main result is the following:

\begin{theoremB}[see Corollary \ref{cor:main}]
Let $X$ be a cocompact Fuchsian building, and let
$\Gamma<\Aut(X)$ be a cocompact lattice. Assume that
$
\Confdim(\partial X)<2.
$
Then
$$H^1(\Gamma,\ell^2(\Gamma)) \neq 0 \quad\text{and}\quad H^2(\Gamma,\ell^2(\Gamma))=0.$$
Consequently, every skew-adjoint $n$-tuple of elements in
$
Z^1\bigl(\Gamma,\HS(\ell^2(\Gamma))_{\Ad\lambda}\bigr)
$
integrates to an $n$-parameter analytic unitary Hilbert-Schmidt
deformation of the left regular representation of $\Gamma$.
\end{theoremB}

The paper is organized as follows. In Section \ref{sec:motivation}, we start out explaining a classical construction of Pytlik and Szwarc \cite{PS86} and motivate the study of analytic deformations. Section~\ref{sec:frechet} discusses the Fr\'echet topology on cochains and applies the open-mapping lemma in order to obtain uniform estimates for primitives of cocycles. Section~\ref{sec:formal} treats the formal unitary obstruction theory. Section~\ref{sec:analytic} proves analytic unitary integration. Section~\ref{sec:applications} records examples and applications.

\section{A prototype deformation for free groups}
\label{sec:motivation}

We record a concrete prototype for the deformation-theoretic picture considered in this article.
It goes back to the analytic family of unitary representations of free groups constructed by Pytlik and Szwarc \cite{PS86}; a particularly transparent realization on a fixed Hilbert space was given by Szwarc \cite{Sz88}.
The surrounding circle of ideas is closely related to the use of tree cocycles in the work of Pimsner--Voiculescu \cite{PV82}, Pimsner \cite{Pi87}, Cuntz \cite{Cu83}, and Julg--Valette \cite{JV84}.

Let $F_k$ be the free group on $k\ge 2$ generators, and let
$
\lambda \colon F_k \to \mathcal U(\ell^2(F_k))
$
be the left regular representation.
Following \cite{Sz88}, define an operator $P$ on $\ell^2(F_k)$ by
\[
P\delta_x=\delta_{x'} \quad (x\neq e), \qquad P\delta_e=0,
\]
where $x'$ is obtained from the reduced word $x$ by deleting its last letter.
Let $T$ be the orthogonal projection onto $\mathbf C\delta_e$, and for a \emph{real} parameter $t \in (-1,1)$ set
\[
T_t=\sqrt{1-t^2}\,T+\bigl(1_{\ell^2(F_k)}-T\bigr), \qquad A_t=\bigl(1_{\ell^2(F_k)}-tP\bigr)T_t.
\]

By Szwarc's spectral computation, \(r(P)=\sqrt{2k-1}\). Hence, by the
spectral-radius formula, the Neumann series
$
\sum_{m\ge 0} t^m P^m
$
converges in operator norm whenever \(|t|<(2k-1)^{-1/2}\). Thus
\(1-tP\) is invertible in this range.
As $T_t$ is invertible whenever $t^2\neq 1$, it follows that $A_t$ is invertible for
$
|t|<(2k-1)^{-1/2}.
$
In this range one obtains a one-parameter analytic family of unitary representations
\[
\pi_t(g)=A_t^{-1}\lambda(g)A_t \qquad (g\in F_k).
\]
By construction, $\pi_0=\lambda$. This construction can be extended to $t \in [-1,1]$ and as uniformly bounded representations even to complex parameters. It is particularly interesting for $|t| \geq (2k-1)^{-1/2}$, however, for our purposes, only the behavior at $t=0$ matters.
Since
$
A_0=1_{\ell^2(F_k)},  T'_0=0, A'_0=-P,
$
differentiating the identity
$
\pi_t(g)=A_t^{-1}\lambda(g)A_t
$
at $t=0$ yields
$
\dot\pi_0(g)=P\lambda(g)-\lambda(g)P=[P,\lambda(g)].
$
Hence the associated infinitesimal cocycle is
$
c(g):=\dot\pi_0(g)\lambda(g)^{-1}
      =P-\lambda(g)P\lambda(g)^{-1}.
$
This is a $1$-cocycle for the unitary coefficient representation
\[
\Ad_\lambda \colon F_k \to \mathcal U\bigl({\rm HS}(\ell^2(F_k))\bigr),\qquad
(\Ad_\lambda)(g)(X)=\lambda(g)X\lambda(g)^{-1},
\]
since
$
c(gh)=c(g)+(\Ad_\lambda)(g)c(h).
$

Moreover, $c(g)$ is in fact finite rank for every $g\in F_k$.
Indeed, for a free generator $a$ one checks directly that $[P,\lambda(a)]$ vanishes on all basis vectors except $\delta_e$ and $\delta_{a^{-1}}$, so it has rank at most $2$.
If $g=a_1\cdots a_n$ is a reduced word, then
\[
[P,\lambda(g)]
=
\sum_{j=1}^n
\lambda(a_1\cdots a_{j-1})[P,\lambda(a_j)]\lambda(a_{j+1}\cdots a_n),
\]
and therefore $[P,\lambda(g)]$ is finite rank.
Consequently,
$
c(g)\in {\rm HS}(\ell^2(F_k))$ for all $g\in F_k.
$

Thus, for $t$ small enough, the Pytlik-Szwarc family provides an explicit analytic deformation of the regular representation on a fixed Hilbert space, whose first derivative is a concrete Hilbert-Schmidt valued cocycle for the adjoint of the regular representation. Note that the operator $P$ is not itself Hilbert-Schmidt, so the deformation is genuinely nontrivial.

Note that $P$ is \emph{not} skew-adjoint as one might expect for the first derivative of a deformation of a unitary representation. However, $P+P^*$ lies in the commutant of the left regular representation. Similarly, $A_t$ is not unitary, but $A_tA_t^*$ is in the commutant of $\lambda(F_k)$. Our aim is to generalize this construction to the broader setting. Indeed, our starting point in the general setting is  a $1$-cocycle $c\in Z^1(G,\HS(\mathcal H)_{\Ad\lambda})$. We ask whether there exists an analytic family of unitary representations $\lambda_t$ with $\lambda_0=\lambda$ and $(d/dt) \lambda_t(g)|_{t=0}=c(g)\lambda(g)$ for all $g\in G$. As mentioned in the introduction, the main result of this paper is that if the second cohomology group with coefficients in $\ell^2(G)$ vanishes, then every $n$-tuple of skew-adjoint cocycles in $Z^1(G,\HS(\mathcal H)_{\Ad\lambda})$ can be integrated to an $n$-parameter analytic unitary deformation of $\lambda$.

\medskip

There is a direct construction of such deformations in the diagonal case, which
avoids the obstruction-theoretic argument entirely.
Let $\lambda\colon G\to U(\ell^2(G))$ be the left regular representation,
and let
$
b\colon G\to \ell^2(G)
$
be a $1$-cocycle for the left translation action
\[
b(gh)=b(g)+g\cdot b(h),
\qquad
(g\cdot \xi)(x)=\xi(g^{-1}x).
\]
Embed $\ell^2(G)$ into $\HS(\ell^2(G))$ by diagonal multiplication:
for $\xi\in \ell^2(G)$ let $M_\xi$ be given by
$
M_\xi \delta_x=\xi(x)\delta_x.
$
Then $M_\xi\in \HS(\ell^2(G))$, and
$
\lambda(g)M_\xi\lambda(g)^{-1}=M_{g\cdot \xi}.
$
Hence
$
c(g):=M_{b(g)}
$
is a $1$-cocycle with values in $\HS(\ell^2(G))_{\mathrm{Ad}\lambda}$.

Assume now that $b(g)$ is purely imaginary for every $g\in G$. Then
$M_{b(g)}$ is skew-adjoint, and for every real $t$ we may define
$
u_t(g):=\exp(M_{t\,b(g)})=M_{\exp(t\,b(g))}.
$
Since diagonal multiplication operators commute, the cocycle identity for $b$
implies
\[
u_t(gh)
=
M_{\exp(t\,b(gh))}
=
M_{\exp(t\,b(g))}\,M_{\exp(t\,g\cdot b(h))}
=
u_t(g)\,(g\cdot u_t(h)).
\]
Therefore
$
\pi_t(g):=u_t(g)\lambda(g)
$
is a unitary representation for every real $t$.
Moreover,
$
u_t(g)-1_{\ell^2(G)}=M_{e^{t\,b(g)}-1}\in \HS(\ell^2(G)),
$
since $b(g)\in \ell^2(G)\subset \ell^\infty(G)$ and
$
|e^{t\,b(g)(x)}-1|\le |t|\,|b(g)(x)|
$
for purely imaginary $b(g)(x)$. Thus each $u_t(g)$ is of the form
$
u_t(g) \in 1_{\ell^2(G)}+ \HS(\ell^2(G))$.
Finally,
$
(d/dt)|_{t=0}u_t(g)=M_{b(g)},
$
so this deformation integrates the given skew-adjoint cocycle.
In particular, for cocycles arising from the diagonal copy of
$\ell^2(G)$, the deformation can be written down explicitly. The same argument applies also to $n$-tuples of cocycles in $Z^1(G,\ell^2(G))$ viewed as taking values in diagonal Hilbert-Schmidt operators or any other self-adjoint abelian subalgebra.

It is also transparent that inner cocycles can always be integrated, see Proposition \ref{prop:inner}, so the focus is on the obstructions arising from non-trivial cohomology classes in $H^1(G,\HS(\ell^2(G)))$, where the values of the  cocycles do not commute as Hilbert-Schmidt operators.

\section{Fr\'echet cochains and \texorpdfstring{$\ell^2$}{ell2}-cohomology}
\label{sec:frechet}
Let $G$ be a countable discrete group and let $V$ be an isometric Banach $G$-module. For each $n\ge 0$ we write $C^n(G,V)$ for the inhomogeneous $n$-cochains.

\begin{definition}
We equip $C^n(G,V)=V^{G^n}$ with the locally convex topology of pointwise convergence. A defining family of seminorms of this topology is given by
$
p_{E}(c)=\max_{x\in E}\|c(x)\|_V
$
where $E$ runs through the finite subsets of $G^n$. Since $G^n$ is countable and $V$ is Banach, the space $C^n(G,V)$ is a Fr\'echet space.
\end{definition}

The differential $d_n\colon C^n(G,V)\to C^{n+1}(G,V)$ is continuous. Therefore the cocycle space $Z^n(G,V)=\ker d_n$ is closed and Fr\'echet, and the quotient $C^n(G,V)/Z^n(G,V)$ is again Fr\'echet. We use unreduced cohomology throughout:
\[
H^n(G,V)=Z^n(G,V)/B^n(G,V),
\qquad
B^n(G,V)=d_{n-1}C^{n-1}(G,V).
\]

\begin{lemma}\label{lem:openmapping}
Let $V$ be a Fr\'echet $G$-module and let $n\ge 0$. Assume that $H^{n+1}(G,V)=0$. Then the induced map
\[
\bar d_n\colon C^n(G,V)/Z^n(G,V)\to Z^{n+1}(G,V)
\]
is an isomorphism of Fr\'echet spaces. In particular, for every finite set $E\subset G^n$ there exist a finite set $F\subset G^{n+1}$ and a constant $\kappa>0$ such that every cocycle $z\in Z^{n+1}(G,V)$ admits a primitive $c\in C^n(G,V)$ with
\[
d_nc=z,
\qquad
p_E(c)\le \kappa p_F(z).
\]
\end{lemma}

\begin{proof}
The assumption $H^{n+1}(G,V)=0$ means that every $(n+1)$-cocycle is a coboundary, so $d_n\colon C^n(G,V)\to Z^{n+1}(G,V)$ is surjective. Since $\ker d_n=Z^n(G,V)$, it factors through a continuous linear bijection
\[
\bar d_n\colon C^n(G,V)/Z^n(G,V)\to Z^{n+1}(G,V).
\]
Both spaces are Fr\'echet. By the open mapping theorem for Fr\'echet spaces, continuous linear surjections between Fr\'echet spaces are open; see Schaefer--Wolff~\cite[Thm.~III.4.2]{SchaeferWolff1999}. Hence $\bar d_n$ is a topological isomorphism.

For a finite set $E\subset G^n$, define the quotient seminorm
\[
\overline p_E([c])=\inf_{u\in Z^n(G,V)} p_E(c+u).
\]
Because $\bar d_n^{-1}$ is continuous, there exist a finite set $F\subset G^{n+1}$ and a constant $\kappa>0$ such that
\[
\overline p_E\bigl(\bar d_n^{-1}(z)\bigr)\le \kappa p_F(z)
\qquad (z\in Z^{n+1}(G,V)).
\]
By the definition of the quotient seminorm, one can choose a representative $c$ of the class $\bar d_n^{-1}(z)$ with
$
d_nc=z,$ and $
p_E(c)\le 2\kappa p_F(z).
$
After enlarging $\kappa$, this proves the claimed estimate.
\end{proof}

Let $\mathcal H$ be a separable Hilbert space. The Hilbert-Schmidt ideal $\HS(\mathcal H)$ consists of all operators $T\in \BH$ such that
$
\|T\|_2^2=\operatorname{Tr}(T^*T)<\infty.
$
It is a Hilbert space for the inner product $\langle S,T\rangle=\operatorname{Tr}(S^*T)$ and satisfies the ideal estimate
\[
\|XTY\|_2\le \|X\|\,\|T\|_2\,\|Y\| \quad (X,Y\in \BH,\,T\in \HS(\mathcal H)).
\]

If $\pi\colon G\to U(\mathcal H)$ is unitary, then
$
\Ad_\pi(g)(T)=\pi(g)T\pi(g)^{-1}
$
defines an isometric action of $G$ on $\HS(\mathcal H)$. We write $\HS(\mathcal H)_{\Ad\pi}$ for this coefficient module. Fell's absorption principle yields the following lemma:

\begin{lemma}[Fell]\label{lem:HSdecomp}
There is an isometric isomorphism of Hilbert $G$-modules
\[
\HS(\ell^2 G)_{\Ad\La}\cong \ell^2(G)\otimes \ol{\ell^2(G)}\cong \ell^2(G)\otimes \mathcal K,
\]
where $\mathcal K$ is a separable Hilbert space with trivial $G$-action.
\end{lemma}

\begin{proof}
The standard matrix-coefficient identification gives an isometric isomorphism
$
\HS(\ell^2 G)\cong \ell^2(G)\otimes \ol{\ell^2(G)}.
$
Under this identification, conjugation by $\La(g)$ sends $\delta_x\otimes \ol{\delta_y}$ to $\delta_{gx}\otimes \ol{\delta_{gy}}$. Conjugating by the unitary map
$
U(\delta_x\otimes \ol{\delta_y})=\delta_{x}\otimes \ol{\delta_{x^{-1}y}}
$
transforms the action into $\La\otimes 1$. Since $\ol{\ell^2(G)}$ is a separable Hilbert space, it is isomorphic to some fixed separable Hilbert space $\mathcal K$.
\end{proof}
Let $G$ be a group.
The following lemma is standard:

\begin{lemma}\label{lem:transfer}
Let $G$ be finitely generated. Then
\[
H^2\bigl(G,\HS(\ell^2 G)_{\Ad\La}\bigr)=0
\quad\Longleftrightarrow\quad
H^2\bigl(G,\ell^2(G)\bigr)=0.
\]
\end{lemma}

\begin{proof}
By Lemma~\ref{lem:HSdecomp}, it suffices to compare
$\ell^2(G)\otimes\mathcal K$ with $\ell^2(G)$, where
$\mathcal K$ has trivial action. Fix an orthonormal basis
$(e_j)_{j\ge 1}$ of $\mathcal K$.

Assume first that
$
H^2\bigl(G,\ell^2(G)\bigr)=0.
$
Let
$
c\in Z^2\bigl(G,\ell^2(G)\otimes\mathcal K\bigr)
$
and write
\[
c(g,h)=\sum_{j\ge 1}c_j(g,h)\otimes e_j.
\]
Then each $c_j$ belongs to
$Z^2\bigl(G,\ell^2(G)\bigr)$.

Fix a finite symmetric generating set $S$ of $G$. By
Lemma~\ref{lem:openmapping}, there exist a finite set
$F\subseteq G^2$ and a constant $\kappa>0$ such that, for
every $j$, there is a primitive
$
b_j\in C^1\bigl(G,\ell^2(G)\bigr),$ with $
d_1b_j=c_j,$
satisfying
\[
\max_{s\in S}\|b_j(s)\|_2
\leq
\kappa\max_{(g,h)\in F}\|c_j(g,h)\|_2.
\]
Consequently, for every $s\in S$,
\[
\begin{aligned}
\sum_{j\ge 1}\|b_j(s)\|_2^2
&\leq
\kappa^2
\sum_{j\ge 1}
\max_{(g,h)\in F}\|c_j(g,h)\|_2^2 \\
&\leq
\kappa^2
\sum_{(g,h)\in F}
\sum_{j\ge 1}\|c_j(g,h)\|_2^2 
=
\kappa^2
\sum_{(g,h)\in F}\|c(g,h)\|_2^2
<\infty.
\end{aligned}
\]
Thus
$
b(s):=\sum_{j\ge 1}b_j(s)\otimes e_j
$
is well defined for every $s\in S$.

Now let $g\in G$ and choose a word
$
g=s_1\cdots s_r,$ with $
s_1,\ldots,s_r\in S.$
Since $d_1b_j=c_j$, one has
\[
b_j(g)
=
b_j(s_1)
+
\sum_{i=2}^r
s_1\cdots s_{i-1}\cdot b_j(s_i)
-
\sum_{i=2}^r
c_j(s_1\cdots s_{i-1},s_i).
\]
It follows that
\[
\begin{aligned}
b(g)
:={}&
b(s_1)
+
\sum_{i=2}^r
s_1\cdots s_{i-1}\cdot b(s_i)
-
\sum_{i=2}^r
c(s_1\cdots s_{i-1},s_i)
\end{aligned}
\]
is a well-defined element of
$\ell^2(G)\otimes\mathcal K$. Coordinatewise, it equals
$
\sum_{j\ge 1}b_j(g)\otimes e_j.
$
Hence
$
b\in C^1\bigl(G,\ell^2(G)\otimes\mathcal K\bigr),
$
and the identity $d_1b=c$ follows coordinatewise. Therefore
$
H^2\bigl(G,\ell^2(G)\otimes\mathcal K\bigr)=0.
$

Conversely, choose a basis vector $e_1\in\mathcal K$. The
equivariant maps
$
\iota(v)=v\otimes e_1
$
and
$
P\left(\sum_jv_j\otimes e_j\right)=v_1
$
satisfy $P\circ\iota={\rm id}$. Passing to cohomology shows that
vanishing for $\ell^2(G)\otimes\mathcal K$ implies vanishing
for $\ell^2(G)$.
\end{proof}

\section{Formal unitary deformations}\label{sec:formal}

Fix a unitary representation $\pi\colon G\to U(\mathcal H)$ and consider $\HS(\mathcal H)_{\Ad\pi}$.
Although the previous examples considered deformation with only one parameter, we will make a simple generalization to multiple-parameter deformation. We write $t=(t_1,\dots,t_n) \in \mathbb R^n$ for the deformation parameters and use multi-index notation throughout: for $\alpha=(\alpha_1,\dots,\alpha_n)\in \N^n$ we set
\[
|\alpha|=\alpha_1+\cdots+\alpha_n,
\qquad
t^\alpha=t_1^{\alpha_1}\cdots t_n^{\alpha_n}.
\]
If $\beta,\alpha\in \N^n$, then $\beta\le \alpha$ means $\beta_i\le \alpha_i$ for all $i$, and $0<\beta<\alpha$ means $0\le \beta_i\le \alpha_i$ for all $i$ with $\beta\neq 0$ and $\beta\neq \alpha$. Thus a formal series with coefficients in $\HS(\mathcal H)_{\Ad\pi}$ is an expression of the form $\sum_{\alpha\in \N^n} b_\alpha t^\alpha$ with $b_\alpha\in \HS(\mathcal H)_{\Ad\pi}$, interpreted purely formally, coefficient by coefficient.

\begin{definition}\label{def:formal-unitary-deformation}
An $n$-parameter formal unitary deformation of $\pi$ is a formal series
\[
\pi_t(g)=u_t(g)\pi(g),
\qquad
u_t(g)=1_{\mathcal H}+\sum_{|\alpha|\ge 1}c_\alpha(g)t^\alpha \in B(\mathcal H)[[t_1,\dots,t_n]]
\]
with $c_\alpha(g)\in \HS(\mathcal H)_{\Ad\pi}$, such that the assignment $g\mapsto \pi_t(g)$ is a group homomorphism, in the sense that
$
\pi_t(gh)=\pi_t(g)\pi_t(h)$ for all $g,h\in G$
as an identity of formal series, and such that
$
\pi_t(g)^*\pi_t(g)=1_{\mathcal H}$ for all $
g\in G.
$
Equivalently, since $\pi$ is a representation and $g\cdot T=\pi(g)T\pi(g)^{-1}$, these conditions are
\begin{equation}\label{eq:formal-unitary}
u_t(gh)=u_t(g)\bigl(g\cdot u_t(h)\bigr)
\quad\text{and}\quad
u_t(g)^*u_t(g)=1_{\mathcal H}.
\end{equation}
\end{definition}

These conditions can be expressed explicitly in terms of the coefficients of the power series $u_t$.

\begin{lemma}\label{lem:firstorder}
Let
\[
\pi_t(g)=u_t(g)\pi(g),
\qquad
u_t(g)=1_{\mathcal H}+\sum_{|\alpha|\ge 1}c_\alpha(g)t^\alpha
\]
be a formal unitary deformation of $\pi$. Then, for every $\alpha\ne 0$, the coefficients satisfy the recursion formulas
\begin{equation}\label{eq:formalrecursion}
c_\alpha(gh)=c_\alpha(g)+g\cdot c_\alpha(h)+\sum_{0<\beta<\alpha}c_\beta(g)\,g\cdot c_{\alpha-\beta}(h)
\end{equation}
and
\begin{equation}\label{eq:unitarycoeff}
c_\alpha(g)^*+c_\alpha(g)=-\sum_{0<\beta<\alpha}c_\beta(g)^*c_{\alpha-\beta}(g).
\end{equation}
In particular, every coefficient $c_{e_i}$ is a skew-adjoint $1$-cocycle.
\end{lemma}

\begin{proof}
The identities \eqref{eq:formalrecursion} and \eqref{eq:unitarycoeff} are obtained by expanding \eqref{eq:formal-unitary}
in degree $\alpha\ne 0$.
For $|\alpha|=1$ the quadratic sums disappear. Therefore
$
c_{e_i}(gh)=c_{e_i}(g)+g\cdot c_{e_i}(h)
$
and
$
c_{e_i}(g)^*+c_{e_i}(g)=0.
$
So $c_{e_i}\in Z^1(G,\HS(\mathcal H)_{\Ad\pi})$ and is skew-adjoint.
\end{proof}

\begin{lemma}\label{lem:obstruction}
Assume that coefficients $c_\beta\in C^1(G,\HS(\mathcal H)_{\Ad\pi})$ are given for all $0<|\beta|<m$ and satisfy \eqref{eq:formalrecursion} in lower degrees. For $|\alpha|=m$ define
\[
r_\alpha(g,h)=-\sum_{0<\beta<\alpha}c_\beta(g)\,g\cdot c_{\alpha-\beta}(h).
\]
Then $r_\alpha\in Z^2(G,\HS(\mathcal H)_{\Ad\pi})$. Moreover, a cochain $c_\alpha\in C^1(G,\HS(\mathcal H)_{\Ad\pi} )$ satisfies \eqref{eq:formalrecursion} in degree $\alpha$ if and only if
$
d_1c_\alpha=r_\alpha.
$
\end{lemma}

\begin{proof}
The equivalence with the degree-$\alpha$ recursion is just a rearrangement of \eqref{eq:formalrecursion}. It remains to compute $d_2r_\alpha$.

For $g,h,k\in G$ one has
\begin{align*}
(d_2r_\alpha)(g,h,k)
&=g\cdot r_\alpha(h,k)-r_\alpha(gh,k)+r_\alpha(g,hk)-r_\alpha(g,h)\\
&=-\sum_{0<\beta<\alpha}g\cdot c_\beta(h)\,gh\cdot c_{\alpha-\beta}(k)
+\sum_{0<\beta<\alpha}c_\beta(gh)\,gh\cdot c_{\alpha-\beta}(k)\\
&\phantom{={}}-\sum_{0<\beta<\alpha}c_\beta(g)\,g\cdot c_{\alpha-\beta}(hk)
+\sum_{0<\beta<\alpha}c_\beta(g)\,g\cdot c_{\alpha-\beta}(h).
\end{align*}
Using the lower-degree recursion for $c_\beta(gh)$ and $c_{\alpha-\beta}(hk)$, this becomes
\begin{align*}
(d_2r_\alpha)(g,h,k)
&=\sum_{0<\beta<\alpha}\sum_{0<\gamma<\beta}
c_\gamma(g)\,g\cdot c_{\beta-\gamma}(h)\,gh\cdot c_{\alpha-\beta}(k)\\
&\phantom{={}}-\sum_{0<\beta<\alpha}\sum_{0<\delta<\alpha-\beta}
c_\beta(g)\,g\cdot c_\delta(h)\,gh\cdot c_{\alpha-\beta-\delta}(k).
\end{align*}
In the first double sum put
$
\rho=\gamma,
\sigma=\beta-\gamma,
\tau=\alpha-\beta,
$
and in the second put
$
\rho=\beta,
\sigma=\delta,
\tau=\alpha-\beta-\delta.
$
Both sums are indexed by all triples $(\rho,\sigma,\tau)$ of nonzero multi-indices satisfying $\rho+\sigma+\tau=\alpha$, and both have the same summand
$
c_\rho(g)\,g\cdot c_\sigma(h)\,gh\cdot c_\tau(k)
$
with opposite signs. Hence the terms cancel pairwise, so $d_2r_\alpha=0$.
\end{proof}

\begin{lemma}\label{lem:formalcorrection}
Assume that coefficients $c_\beta\in C^1(G,\HS(\mathcal H)_{\Ad\pi})$ are given for all $0<|\beta|<m$ and satisfy both \eqref{eq:formalrecursion} and \eqref{eq:unitarycoeff} in lower degrees. Let $|\alpha|=m$, and let $b_\alpha\in C^1(G,\HS(\mathcal H)_{\Ad\pi})$ satisfy
$
d_1b_\alpha=r_\alpha,
$
where $r_\alpha$ is the obstruction cocycle from Lemma~\ref{lem:obstruction}. Define
\[
s_\alpha(g)=b_\alpha(g)^*+b_\alpha(g)+\sum_{0<\beta<\alpha}c_\beta(g)^*c_{\alpha-\beta}(g).
\]
Then $s_\alpha\in Z^1(G,\HS(\mathcal H)_{\Ad\pi})$, and
$
c_\alpha=b_\alpha-\frac12 s_\alpha
$
satisfies both \eqref{eq:formalrecursion} and \eqref{eq:unitarycoeff} in degree $\alpha$.
\end{lemma}

\begin{proof}
Because $d_1b_\alpha=r_\alpha$, the coefficient $b_\alpha$ satisfies \eqref{eq:formalrecursion} in degree $\alpha$. The coefficient of degree $\alpha$ in
\[
u_t(gh)^*u_t(gh)-\bigl(g\cdot u_t(h)\bigr)^*u_t(g)^*u_t(g)\bigl(g\cdot u_t(h)\bigr)
\]
is exactly
$
s_\alpha(gh)-s_\alpha(g)-g\cdot s_\alpha(h).
$
Since the lower-order coefficients already satisfy the deformation and unitarity equations, this expression vanishes, so $s_\alpha\in Z^1(G,\HS(\mathcal H)_{\Ad\pi})$. Therefore $d_1c_\alpha=d_1b_\alpha=r_\alpha$, so \eqref{eq:formalrecursion} still holds in degree $\alpha$. Moreover,
\[
c_\alpha(g)^*+c_\alpha(g)=b_\alpha(g)^*+b_\alpha(g)-s_\alpha(g),
\]
which is exactly \eqref{eq:unitarycoeff} in degree $\alpha$.
\end{proof}

\begin{corollary}\label{cor:formalexistence}
If $H^2(G,\HS(\mathcal H)_{\Ad\pi})=0$, then every skew-adjoint $n$-tuple in $Z^1(G,\HS(\mathcal H)_{\Ad\pi})^n$ extends to a formal unitary deformation.
\end{corollary}

\begin{proof}
Start with the prescribed skew-adjoint $1$-cocycles. Suppose coefficients of total degree less than $m$ have been constructed and satisfy both recursion relations. For each $|\alpha|=m$, Lemma~\ref{lem:obstruction} produces a cocycle $r_\alpha\in Z^2(G,\HS(\mathcal H)_{\Ad\pi})$. Since $H^2(G,\HS(\mathcal H)_{\Ad\pi})=0$, choose $b_\alpha\in C^1(G,\HS(\mathcal H)_{\Ad\pi})$ with $d_1b_\alpha=r_\alpha$. Lemma~\ref{lem:formalcorrection} then replaces $b_\alpha$ by a coefficient $c_\alpha$ satisfying both recursion relations in degree $\alpha$. Proceeding inductively yields the full formal unitary deformation.
\end{proof}

\begin{remark}
The obstruction classes produced in Lemma~\ref{lem:obstruction} may be viewed as
Massey products of the initial classes in $H^1(G,\HS(\mathcal H)_{\Ad\pi})$.
In degree two these reduce to ordinary cup products. In higher degree,
after choosing defining systems in lower orders, the cocycles $r_\alpha$
encode the corresponding higher-order Massey products. Thus
Corollary~\ref{cor:formalexistence} may be interpreted as saying that the vanishing of
$H^2(G,\HS(\mathcal H)_{\Ad\pi})$ forces all these Massey-type obstruction classes to vanish,
so every skew-adjoint first-order datum extends to a formal unitary
deformation.
\end{remark}

\section{Analytic unitary deformations}\label{sec:analytic}
Fix a finite symmetric generating set $S$ of $G$ and define
$
\|c\|_S=\max_{s\in S}\|c(s)\|_2$ for $c\in C^1(G,\HS(\mathcal H)_{\Ad\pi}).
$
Because the action on $B$ is isometric, one has
\[
\|x\,g\cdot y\|_2\le \|x\|_2\,\|y\|_2
\qquad (x,y\in \HS(\mathcal H)_{\Ad\pi}).
\]

The purpose of this section is to show that the formal unitary deformations constructed in Corollary~\ref{cor:formalexistence} are actually analytic. To do so, we impose a growth condition on the coefficients of the power series $u_t$ and show that it can be preserved by the recursive construction of solutions to the obstruction equations. Accordingly, we start with skew-adjoint cocycles $c_{e_i}\in Z^1(G,\HS(\mathcal H)_{\Ad\pi})$ for $1\le i\le n$.

\begin{definition}
Let
$
a=(a_1,\dots,a_n),
a_i=\|c_{e_i}\|_S,
a^\alpha=a_1^{\alpha_1}\cdots a_n^{\alpha_n}.
$
We consider the following bound on the coefficients of $u_t$:
\begin{equation}\label{eq:A2} \tag{$\star_\eta$}
\|c_\alpha\|_S\le \eta^{|\alpha|-1}\frac{1}{|\alpha|}\binom{|\alpha|}{\alpha}C_{|\alpha|-1}a^\alpha
\qquad (\alpha\ne 0),
\end{equation}
where
\[
C_m=\frac{1}{m+1}\binom{2m}{m}
\qquad (m\ge 0)
\]
denotes the $m$-th Catalan number and $\eta>0$ is a constant.
\end{definition}

For later use we record the multinomial Vandermonde identity
\begin{equation}\label{eq:multivandermonde}
\sum_{\substack{\beta\le \alpha\\ |\beta|=j}}
\binom{j}{\beta}\binom{|\alpha|-j}{\alpha-\beta}
=\binom{|\alpha|}{\alpha}
\qquad (0\le j\le |\alpha|).
\end{equation}

\begin{lemma}\label{lem:multicatalan}
Let $\alpha\in \N^n$ with $|\alpha|=m\ge 2$. Then
\[
\sum_{0<\beta<\alpha}
\frac{1}{|\beta|}\binom{|\beta|}{\beta}C_{|\beta|-1}
\frac{1}{|\alpha|-|\beta|}\binom{|\alpha|-|\beta|}{\alpha-\beta}C_{|\alpha|-|\beta|-1}
\le 2\,\frac{1}{m}\binom{m}{\alpha}C_{m-1}.
\]
\end{lemma}

\begin{proof}
For fixed $j\in\{1,\dots,m-1\}$, the conditions $0<\beta<\alpha$ and $|\beta|=j$ are equivalent to $\beta\le \alpha$ and $|\beta|=j$: the cases $\beta=0$ and $\beta=\alpha$ are impossible because $j$ is strictly between $0$ and $m$. Hence, after grouping the sum by $j=|\beta|$ and using \eqref{eq:multivandermonde}, we obtain
\begin{align*}
&\sum_{0<\beta<\alpha}
\frac{1}{|\beta|}\binom{|\beta|}{\beta}C_{|\beta|-1}
\frac{1}{|\alpha|-|\beta|}\binom{|\alpha|-|\beta|}{\alpha-\beta}C_{|\alpha|-|\beta|-1}
\\
&\qquad=
\sum_{j=1}^{m-1}
\frac{C_{j-1}C_{m-j-1}}{j(m-j)}
\sum_{\substack{\beta\le \alpha\\ |\beta|=j}}
\binom{j}{\beta}\binom{m-j}{\alpha-\beta}
\\
&\qquad=
\binom{m}{\alpha}
\sum_{j=1}^{m-1}\frac{C_{j-1}C_{m-j-1}}{j(m-j)}.
\end{align*}
For $1\le j\le m-1$ one has $j(m-j)\ge m/2$, hence
\[
\frac{1}{j(m-j)}\le \frac{2}{m}.
\]
Therefore
\[
\sum_{j=1}^{m-1}\frac{C_{j-1}C_{m-j-1}}{j(m-j)}
\le \frac{2}{m}\sum_{j=1}^{m-1}C_{j-1}C_{m-j-1}
=\frac{2}{m}C_{m-1},
\]
where the last equality is the Catalan recursion. This proves the claim.
\end{proof}

\begin{lemma}\label{lem:finitepropagation}
Let $E\subset G$ be finite, and choose for each $g\in E$ a word in the generators $S$ representing $g$. Assume that the coefficients satisfy \eqref{eq:formalrecursion} and the bound \eqref{eq:A2}. Then there exists a constant $L\ge 1$, depending only on $E$ and the chosen words, such that
\begin{equation}\label{eq:finitepropagation}
\max_{g\in E}\|c_\alpha(g)\|_2
\le L \eta^{|\alpha|-1}\frac{1}{|\alpha|}\binom{|\alpha|}{\alpha}C_{|\alpha|-1}a^\alpha
\qquad (\alpha\ne 0).
\end{equation}
\end{lemma}

\begin{proof}
Write
\[
A_\alpha=\frac{1}{|\alpha|}\binom{|\alpha|}{\alpha}C_{|\alpha|-1}a^\alpha.
\]
It suffices to prove that for each chosen word $g=s_1\cdots s_\ell$ there is a constant $L_g\ge 1$, depending only on that word, such that
\[
\|c_\alpha(g)\|_2\le L_g \eta^{|\alpha|-1}A_\alpha
\qquad (\alpha\ne 0).
\]
We argue by induction on the word length $\ell$.

If $\ell=1$, then $g=s_1\in S$, so the claim follows directly from the definition of $\|c_\alpha\|_S$ with $L_g=1$.

Assume now that $\ell\ge 2$, and write $g=hs$ where $h=s_1\cdots s_{\ell-1}$ and $s=s_\ell\in S$. By the induction hypothesis there is a constant $L_h\ge 1$ such that
\[
\|c_\beta(h)\|_2\le L_h \eta^{|\beta|-1}A_\beta
\qquad (\beta\ne 0).
\]
Using \eqref{eq:formalrecursion}, the isometricity of the action, the definition of $\|c_\alpha\|_S$, and Lemma~\ref{lem:multicatalan}, we obtain for $\alpha\ne 0$
\begin{align*}
\|c_\alpha(g)\|_2
&\le \|c_\alpha(h)\|_2+\|h\cdot c_\alpha(s)\|_2
+\sum_{0<\beta<\alpha}\|c_\beta(h)\|_2\,\|h\cdot c_{\alpha-\beta}(s)\|_2\\
&\le L_h \eta^{|\alpha|-1}A_\alpha+\eta^{|\alpha|-1}A_\alpha
+L_h \eta^{|\alpha|-2}a^\alpha
\\
&\qquad
\sum_{0<\beta<\alpha}
\frac{1}{|\beta|}\binom{|\beta|}{\beta}C_{|\beta|-1}
\frac{1}{|\alpha|-|\beta|}\binom{|\alpha|-|\beta|}{\alpha-\beta}C_{|\alpha|-|\beta|-1}\\
&\le (L_h+1+2L_h)\eta^{|\alpha|-1}A_\alpha.
\end{align*}
Hence the claim holds for $g$ with $L_g=1+3L_h$.

Taking the maximum of the finitely many constants $L_g$ over $g\in E$ proves the lemma.
\end{proof}

\begin{lemma}\label{lem:onestep}
Assume that $H^2(G,\HS(\mathcal H)_{\Ad\pi})=0$. Then there exists a constant $K\ge 1$ with the following property. Let $\eta\ge K$, and assume that coefficients of total degree less than $m$ satisfy \eqref{eq:formalrecursion}, \eqref{eq:unitarycoeff}, and the bound \eqref{eq:A2}. Then for every multi-index $\alpha$ with $|\alpha|=m$ there exists a coefficient $c_\alpha\in C^1(G,\HS(\mathcal H)_{\Ad\pi})$ such that the deformation and unitarity equations continue to hold in degree $\alpha$ and the bound \eqref{eq:A2} holds for $\alpha$.
\end{lemma}

\begin{proof}
Write
\[
A_\alpha=\frac{1}{|\alpha|}\binom{|\alpha|}{\alpha}C_{|\alpha|-1}a^\alpha.
\]
Fix $|\alpha|=m$. By Lemma~\ref{lem:obstruction}, the obstruction $r_\alpha$ is a $2$-cocycle. Since $H^2(G,\HS(\mathcal H)_{\Ad\pi})=0$, Lemma~\ref{lem:openmapping} applied with $n=1$ and $E=S$ yields a finite set $F\subset G^2$, a constant $\kappa>0$, and a cochain $b_\alpha\in C^1(G,\HS(\mathcal H)_{\Ad\pi})$ such that
\[
d_1b_\alpha=r_\alpha,
\qquad
\|b_\alpha\|_S\le \kappa\max_{(g,h)\in F}\|r_\alpha(g,h)\|_2.
\]

To estimate $b_\alpha$ on the generators, we use this choice of $F$ and $\kappa$:
\[
\|b_\alpha\|_S\le \kappa\max_{(g,h)\in F}\|r_\alpha(g,h)\|_2.
\]
Let $E_F$ be the finite set of all coordinates appearing in $F$. Applying Lemma~\ref{lem:finitepropagation} to $E_F$, with its associated constant denoted by $L$, and using the definition of $r_\alpha$, we obtain
\begin{align*}
\max_{(g,h)\in F}\|r_\alpha(g,h)\|_2
&\le \sum_{0<\beta<\alpha}
\max_{g\in E_F}\|c_\beta(g)\|_2\,
\max_{h\in E_F}\|c_{\alpha-\beta}(h)\|_2\\
&\le L^2 \eta^{|\alpha|-2}a^\alpha
\sum_{0<\beta<\alpha}
\frac{1}{|\beta|}\binom{|\beta|}{\beta}C_{|\beta|-1}
\\
&\phantom{\le L^2 \eta^{|\alpha|-2}a^\alpha\sum_{0<\beta<\alpha}}
\cdot
\frac{1}{|\alpha|-|\beta|}\binom{|\alpha|-|\beta|}{\alpha-\beta}C_{|\alpha|-|\beta|-1}.
\end{align*}
Lemma~\ref{lem:multicatalan} bounds the final sum by
$
2A_\alpha.
$
Therefore
\[
\|b_\alpha\|_S\le 2\kappa L^2 \eta^{|\alpha|-2}A_\alpha.
\]
Set $M_1=2\kappa L^2$.
Now define the unitary defect
\[
s_\alpha(g)=b_\alpha(g)^*+b_\alpha(g)+\sum_{0<\beta<\alpha}c_\beta(g)^*c_{\alpha-\beta}(g).
\]
The coefficient of degree $\alpha$ in
\[
u_t(gh)^*u_t(gh)-\bigl(g\cdot u_t(h)\bigr)^*u_t(g)^*u_t(g)\bigl(g\cdot u_t(h)\bigr)
\]
is exactly
$
s_\alpha(gh)-s_\alpha(g)-g\cdot s_\alpha(h).
$
Since the lower-order coefficients already define a unitary deformation, the whole expression vanishes, so $s_\alpha\in Z^1(G,\HS(\mathcal H)_{\Ad\pi})$.

For $g\in S$, the induction hypothesis and Lemma~\ref{lem:multicatalan} give
\begin{align*}
\|s_\alpha(g)\|_2
&\le 2\|b_\alpha(g)\|_2+\sum_{0<\beta<\alpha}\|c_\beta(g)\|_2\,\|c_{\alpha-\beta}(g)\|_2\\
&\le 2M_1 \eta^{|\alpha|-2}A_\alpha
+\eta^{|\alpha|-2}a^\alpha
\sum_{0<\beta<\alpha}
\frac{1}{|\beta|}\binom{|\beta|}{\beta}C_{|\beta|-1}
\\
&\phantom{\le 2M_1 \eta^{|\alpha|-2}A_\alpha
+\eta^{|\alpha|-2}a^\alpha\sum_{0<\beta<\alpha}}
\frac{1}{|\alpha|-|\beta|}\binom{|\alpha|-|\beta|}{\alpha-\beta}C_{|\alpha|-|\beta|-1}\\
&\le (2M_1+2)\eta^{|\alpha|-2}A_\alpha.
\end{align*}
Hence
$
\|s_\alpha\|_S\le (2M_1+2)\eta^{|\alpha|-2}A_\alpha.
$

Set
$
c_\alpha=b_\alpha-\frac12 s_\alpha.
$
Because $s_\alpha$ is a $1$-cocycle, one still has $d_1c_\alpha=r_\alpha$, so the multiplicativity equation is preserved. On the other hand,
$
c_\alpha(g)^*+c_\alpha(g)=b_\alpha(g)^*+b_\alpha(g)-s_\alpha(g),
$
so \eqref{eq:unitarycoeff} now holds in degree $\alpha$. Therefore
\[
\|c_\alpha\|_S\le \|b_\alpha\|_S+\frac12\|s_\alpha\|_S\le (2M_1+1)\eta^{|\alpha|-2}A_\alpha.
\]
Set
$
K=\max\{1,2M_1+1\}.
$
Then for every $\eta\ge K$ we obtain
$
\|c_\alpha\|_S\le \eta^{|\alpha|-1}A_\alpha,
$
which is exactly the required estimate.
\end{proof}

\begin{theorem}\label{thm:unitaryanalytic} Let $G$ be a finitely generated group.
Assume that $c_{e_1},\dots,c_{e_n}\in Z^1(G,\HS(\mathcal H)_{\Ad\pi})$ are skew-adjoint and that $H^2(G,\HS(\mathcal H)_{\Ad\pi})=0$. Then they integrate to an analytic unitary deformation of $\pi$, i.e. there exist a constant $\eta\ge 1$ and a formal unitary deformation
\[
\pi_t(g)=\Bigl(1_{\mathcal H}+\sum_{|\alpha|\ge 1}c_\alpha(g)t^\alpha\Bigr)\pi(g)
\]
whose coefficients satisfy the bound \eqref{eq:A2}. Consequently the series converges absolutely on the open set
\[
\Bigl\{t\in {\mathbb R}^n: a_1|t_1|+\cdots+a_n|t_n|<\frac1{4\eta}\Bigr\}.
\]
The resulting limit defines an analytic unitary deformation.
\end{theorem}

\begin{proof}
Let $K$ be the constant from Lemma~\ref{lem:onestep}, and choose $\eta\ge K$. For $|\alpha|=1$ the prescribed cocycles satisfy the bound \eqref{eq:A2} because
$
\|c_{e_i}\|_S=a_i,
$  
and here $a^{e_i}=a_i$ while the zeroth Catalan number equals $1$. Thus
\[
\|c_{e_i}\|_S=\eta^{0}\frac{1}{1}\binom{1}{e_i}C_{0}a^{e_i}.
\]
Assume inductively that coefficients of total degree less than $m$ have been constructed and satisfy the deformation equations, the unitary equations, and the bound \eqref{eq:A2} with this fixed constant $\eta$. Lemma~\ref{lem:onestep} then produces the coefficients of total degree $m$ with the same properties and the same constant $\eta$. This yields coefficients satisfying the deformation equations, the unitary equations, and the bound \eqref{eq:A2} in every degree.

Summing over all multi-indices of total degree $m$ and using the multinomial theorem gives
\[
\sum_{|\alpha|=m}\binom{m}{\alpha}a^\alpha|t|^\alpha
=(a_1|t_1|+\cdots+a_n|t_n|)^m.
\]
Hence
\[
\sum_{|\alpha|\ge 1}\|c_\alpha\|_S|t|^\alpha
\le \frac{1}{\eta}\sum_{m\ge 1}\frac{1}{m}C_{m-1}\bigl(\eta(a_1|t_1|+\cdots+a_n|t_n|)\bigr)^m.
\]
The generating series
\[
\sum_{m\ge 1}C_{m-1}r^m=\frac{1-\sqrt{1-4r}}{2}
\]
has radius of convergence $1/4$, so the right-hand side converges whenever $\eta(a_1|t_1|+\cdots+a_n|t_n|)<1/4$. This is exactly the simplex condition in the statement. This shows convergence on the generators; the corresponding bound holds for all $g\in G$ by multiplicativity. Therefore the series converges absolutely and uniformly on compact subsets of the open simplex, and the limit is a continuous function of $t$.
Since the defining identities hold coefficientwise, the limit is an analytic unitary deformation.
\end{proof}

The interesting deformations obtained from Theorem~\ref{thm:unitaryanalytic} are those for which the initial cocycles $c_{e_i}$ are not inner. Let us briefly discuss the inner case when $n=1$, which is somewhat degenerate but still worth understanding. In this case the deformation can be integrated explicitly, without using the vanishing of $H^2(G,\HS(\mathcal H)_{\Ad\pi})$.

\begin{proposition}
\label{prop:inner}
Let $\pi\colon G\to U(\mathcal H)$ be a unitary representation. Suppose that
$c\in Z^1(G,\HS(\mathcal H))$ is skew-adjoint and inner, so that there exists
$X\in \HS(\mathcal H)$ with $X^*=-X$ and
\[
c(g)=X-g\cdot X=X-\pi(g)X\pi(g)^{-1}
\qquad (g\in G).
\]
Then $c$ integrates to an analytic one-parameter unitary deformation of $\pi$.
More precisely, if
$
U_t=e^{tX},
$
then $U_t$ is unitary for every $t\in \mathbb R$, and
$
\pi_t(g)=U_t\,\pi(g)\,U_t^{-1}
$
defines an analytic unitary deformation of $\pi$. Moreover,
$
U_t- 1_{\mathcal H} \in \HS(\mathcal H),
$
and the linear term of the deformation is the given cocycle $c$.
\end{proposition}

\begin{proof}
Since $X\in \HS(\mathcal H)\subset B(\mathcal H)$ is bounded, the exponential series
\[
U_t=e^{tX}=\sum_{m\ge 0}\frac{t^mX^m}{m!}
\]
converges in operator norm for all $t\in\mathbb C$. Because $X^*=-X$, the
operator $U_t$ is unitary for real $t$. Also,
\[
K_t:=U_t-1_{\mathcal H}=\sum_{m\ge 1}\frac{t^mX^m}{m!}\in \HS(\mathcal H),
\]
since $\HS(\mathcal H)$ is a two-sided ideal in $B(\mathcal H)$ and every power $X^m$ with
$m\ge 1$ lies in $\HS(\mathcal H)$.
Now define
$
\pi_t(g)=U_t\pi(g)U_t^{-1}.
$
This is a unitary representation for every real $t$, and it is of the form
\[
\pi_t(g)=u_t(g)\pi(g),
\qquad
u_t(g)=U_t\,(g\cdot U_t^{-1}),
\]
so it is a deformation in the sense of Definition~\ref{def:formal-unitary-deformation}.
To identify the first-order term, expand at $t=0$:
$
U_t=1_{\mathcal H}+tX+O(t^2),
$ and $
U_t^{-1}=1_{\mathcal H}-tX+O(t^2).
$
Hence
\[
u_t(g)
=\bigl(1_{\mathcal H}+tX+O(t^2)\bigr)\bigl(1_{\mathcal H}-t\,g\cdot X+O(t^2)\bigr)
=1_{\mathcal H}+t(X-g\cdot X)+O(t^2).
\]
Therefore the coefficient of $t$ is exactly
$
c(g)=X-g\cdot X.
$
So the deformation integrates the prescribed skew-adjoint inner cocycle.
\end{proof}

\section{Applications to cocompact lattices in Fuchsian buildings}
\label{sec:applications}

We now specialize to the left regular representation.

\begin{corollary}\label{cor:regular} Let $G$ be a finitely generated group.
Let $\La\colon G\to U(\ell^2 G)$ be the left regular representation. If $H^2\bigl(G,\ell^2(G)\bigr)=0$, then $n$-tuple of skew-adjoint elements in
$
Z^1(G,\HS(\ell^2 G)_{\Ad\La})
$
integrates to an analytic unitary Hilbert-Schmidt deformation of $\La$.
\end{corollary}

\begin{proof}
By Lemma~\ref{lem:transfer}, the hypothesis implies $H^2\bigl(G,\HS(\ell^2 G)_{\Ad\La}\bigr)=0$. The assertion is therefore an immediate consequence of Theorem~\ref{thm:unitaryanalytic}.
\end{proof}

If $G$ has cohomological dimension at most $1$, then $H^2(G,V)=0$ for every $G$-module $V$; see Brown~\cite[Chap.~VIII]{Brown1982}. In particular, for every free group $F_r$ with $r\ge 1$ one has
$
H^2\bigl(F_r,\ell^2(F_r)\bigr)=0.
$
Therefore Corollary~\ref{cor:regular} applies to free groups, and more generally to all groups of virtual cohomological dimension at most $1$. This case is structurally simpler, since one can essentially deform the generators independently, but it still serves as a useful sanity check and a simple example of the theory.

\medskip

In this section we explain that cocompact lattices in Fuchsian buildings
provide a natural and general source of examples for
Corollary~\ref{cor:regular}. In particular, this covers surface groups. However, the relevant geometric input is the conformal dimension of
the boundary. In the right-angled case this invariant was computed
explicitly by Bourdon \cite{Bou97}, and his vanishing theorem for
unreduced $\ell^p$-cohomology gives a clean sufficient criterion for the
vanishing of the obstruction group $H^2(G,\ell^2(G))$.

We begin with the general statement. The key point is that for a cocompact
lattice $\Gamma$ in a locally finite Fuchsian building $X$, the building $X$
is a cocompact model for $E_{\mathcal F\!in}\Gamma$. Hence \emph{unreduced} group
cohomology with coefficients in a unitary representation is computed by the
corresponding equivariant cochain complex on $X$. 

\begin{theorem}[Bourdon \cite{Bou16}, Bourdon--Pajot \cite{BP03}]\label{thm:fuchsian-building-application}
Let $X$ be a cocompact Fuchsian building, and let
$\Gamma<\Aut(X)$ be a cocompact lattice. Assume that
$
\Confdim(\partial X)<2.
$
Then
$$H^1(\Gamma,\ell^2(\Gamma)) \neq 0 \quad\text{and}\quad H^2(\Gamma,\ell^2(\Gamma))=0.$$

\end{theorem}

\begin{proof}
Since $X$ is a locally finite contractible Fuchsian building and
$\Gamma<\Aut(X)$ is a cocompact lattice, the action of $\Gamma$ on $X$ is
proper, cocompact, and all stabilisers are finite. Moreover, for any finite subgroup $H<\Gamma$, $X^H$ is contractible. Thus $X$ is a model for
$E_{\mathcal F\!in}\Gamma$, the classifying space for the family of finite
subgroups of $\Gamma$.

Because finite subgroups have vanishing higher cohomology with coefficients
in a unitary representation, the equivariant cellular $\ell^2$-cochain complex of $X$
computes the unreduced group cohomology $H^*(\Gamma,\ell^2(\Gamma))$. In particular,
for the left regular representation one obtains
$
H^k(\Gamma,\ell^2(\Gamma))\cong \ell^2 H^k(X),
$
where the right-hand side denotes unreduced $\ell^2$-cohomology of the
building $X$.
Now $X$ is a hyperbolic uniformly contractible complex of bounded geometry.
By Bourdon-Pajot's theorem \cite{BP03} and given the bound on the conformal dimension one has $\ell^2 H^1(X)\ne 0$, so $H^1(\Gamma,\ell^2(\Gamma))\ne 0$. Similarly, by Bourdon's vanishing theorem \cite{Bou16} for unreduced $\ell^p$-cohomology,
\[
\ell^p H^2(X)=0
\qquad\text{for every }p>\Confdim(\partial X).
\]
Since $\Confdim(\partial X)<2$, we may apply this with $p=2$ and obtain
$
\ell^2 H^2(X)=0
$, and hence
$
H^2(\Gamma,\ell^2(\Gamma))=0.
$
\end{proof}

This covers for example surface groups, where the claim can be proved directly, but more interestingly it applies to groups acting on buildings.
In order to provide explicit examples, we now specialise to Bourdon's right-angled Fuchsian buildings $I_{p,q}$. Here $p\ge 5$ and the chambers are regular right-angled
hyperbolic $p$-gons; every edge is contained in exactly $q$ chambers, so
the building is thick for $q\ge 3$.
The boundary conformal dimension is known exactly.

\begin{theorem}[Bourdon, see \cite{Bou97}]\label{thm:bourdon-confdim}
For the Bourdon building $I_{p,q}$ one has
\[
\Confdim(\partial I_{p,q})
=
1+\frac{\log(q-1)}{\arccosh\bigl(\frac{p-2}{2}\bigr)}.
\]
In particular, $\Confdim(\partial I_{p,q})<2$ if and only if
$q\le p-2.$
\end{theorem}

For the right-angled building $I_{p,q}$ there is a standard cocompact
lattice which is particularly concrete. It may be described as the graph
product of $p$ cyclic groups of order $q$ over the cycle graph $C_p$:
\[
\Gamma_{p,q}
=
\Gamma(C_p;C_q,\dots,C_q)
=
\bigl\langle
x_1,\dots,x_p
\ \big|\
x_i^q=1,\ [x_i,x_{i+1}]=1\ \forall\, i\in \mathbb Z/p\mathbb Z
\bigr\rangle.
\]
This group acts chamber-transitively on $I_{p,q}$ as a uniform lattice.
Therefore our result applies in
particular to $\Gamma_{p,q}$ whenever $q\le p-2$.

\begin{corollary} 
\label{cor:main}
Let $X$ be a cocompact Fuchsian building, and let
$\Gamma<\Aut(X)$ be a cocompact lattice. Assume that
$
\Confdim(\partial X)<2,$ for instance if $X=I_{p,q}$ with $q\le p-2$ and $\Gamma=\Gamma_{p,q}$. Then, every $n$-tuple of skew-adjoint elements in
$
Z^1(\Gamma,\HS(\ell^2(\Gamma))_{\Ad\lambda})
$
integrates to an $n$-parameter analytic unitary Hilbert-Schmidt
deformation of the left regular representation of $\Gamma$.
\end{corollary}

The family $I_{p,q}$ is only the most classical example. More generally,
Fuchsian buildings arise from several constructions.
Bourdon and Gaboriau--Paulin constructed many examples as universal covers
of hyperbolic polygons of finite groups; see \cite{Bou00,GP01}. In particular, Gaboriau--Paulin
produce non-right-angled Fuchsian buildings, while Bourdon also discusses
examples with triangular chambers. There are also triangular hyperbolic
buildings constructed by Vdovina and by Carbone--Kangaslampi--Vdovina; see
\cite{Vd02,CKV11}. Further examples come from the hyperbolic Kac--Moody setting; see for instance \cite{Thomas12}. For a
general overview of these constructions and of lattices in hyperbolic
buildings, we refer to Thomas's survey \cite{Thomas12}. 

\begin{remark}
\label{rem:equivalence}
In the case of free groups, see Section \ref{sec:motivation}, it turned out that for small values of the deformation parameter the deformed representations constructed by Pytlik-Szwarc are equivalent to the left-regular representation. While this can be guaranteed for cocycles taking values in $\ell^2(F_k) \otimes \mathbb C[F_k] \subseteq \HS(\ell^2(F_k))$, it is an interesting open question whether this phenomenon persists for arbitrary unitary Hilbert-Schmidt deformations of the left-regular representation of a free group. Moreover, this question is also interesting for surface groups and more generally for the groups studied in this section.
\end{remark}

\section*{Acknowledgements}

The second author thanks Antonio L\'opez Neumann for interesting discussions on the last section. GPT (OpenAI) was used to assist in drafting parts of this manuscript. All content was reviewed and substantially revised by the authors, who are responsible for the final text. The content of this article is also part of the forthcoming doctoral thesis of the first author.

\end{document}